\documentclass[11pt]{amsart}
\usepackage{amssymb}
\usepackage{times}
\usepackage{url}
\usepackage{hyperref}
\newtheorem{theorem}{Theorem}[section]
\newtheorem{lemma}[theorem]{Lemma}
\newtheorem{proposition}[theorem]{Proposition}
\newtheorem{corollary}[theorem]{Corollary}
\theoremstyle{definition}

\newtheorem{example}[theorem]{Example}
\theoremstyle{remark}
\newtheorem{remark}[theorem]{Remark}

\newcommand{\F}{\mathbb{F}}

\newcommand{\Z}{\mathbb{Z}}

\newcommand{\Q}{\mathbb{Q}}

\def\gp#1{\langle \hspace*{.2mm} #1 \hspace*{.15mm} \rangle}
\def\group#1#2{\langle \hspace*{.275mm} #1 \hspace*{1mm} |
\hspace*{1.25mm} #2  \hspace*{.25mm}\rangle}

\begin{document}

\title{Recognizing finite matrix groups over infinite fields}

\author{A. S. Detinko}
\address{School of Mathematics, Statistics and Applied Mathematics,
National University of Ireland, Galway, Ireland}
\email{alla.detinko@nuigalway.ie}

\author{D. L. Flannery}
\address{School of Mathematics, Statistics and Applied Mathematics,
National University of Ireland, Galway, Ireland}
\email{dane.flannery@nuigalway.ie}

\author{E. A. O'Brien}
\address{Department of Mathematics, The University
of Auckland, Private Bag 92019, Auckland, New Zealand}
\email{e.obrien@auckland.ac.nz}



\begin{abstract}
We present a uniform methodology for computing with finitely
generated matrix groups over any infinite field. As one
application, we completely solve the problem of deciding
finiteness in this class of groups. We also present an algorithm
that, given such a finite group as input, in practice successfully
constructs an isomorphic copy over a finite field, and uses this
copy to investigate the group's structure. Implementations of our
algorithms are available in {\sc Magma}.
\end{abstract}

\maketitle
\section{Introduction}

This paper establishes a uniform methodology for computing with
finitely generated linear groups over any infinite field. Our
techniques constitute a computational analogue of `finite
approximation' \cite[Chapter 4]{Wehrfritz}, which is a major tool
in the study of finitely generated linear groups. It relies on the
fact that each finitely generated linear group $G$ is residually
finite. Moreover, $G$ is approximated by matrix groups of the same
degree over finite fields \cite[Theorem A, \mbox{p.}
151]{ZalesskiiSurv1}. We also use the fundamental result that $G$
has a normal subgroup of finite index with every torsion element
unipotent \cite[4.8, \mbox{p.}~56]{Wehrfritz}. For computational
purposes, the key objective is to determine a congruence
homomorphism whose kernel has this property, and whose image is
defined over a finite field.

The first problem that we solve is a natural and obvious candidate
for an application of our methodology: testing finiteness of
finitely generated linear groups. This problem has been
investigated previously, but only for groups over specific
domains. Algorithms for testing finiteness over the rational field
$\Q$ are given in \cite{BBR}. One of these, based on integrality
testing, is exploited as part of the default procedures in {\sf
GAP} \cite{GAP} and {\sc Magma} \cite{Magma} to decide finiteness
over $\Q$. Groups over a characteristic zero function field are
considered in \cite{rtb99}. However, the algorithm there possibly
involves squaring dimensions. Function fields are also dealt with
in \cite{Detinko01,JSC4534,PositiveFiniteness,ivanyos01,rtb99},
where computing in matrix algebras plays a central role. While the
algorithms from \cite{JSC4534,PositiveFiniteness} have been
implemented in {\sc Magma}, we know of no implementations of those
from \cite{Detinko01,ivanyos01,rtb99}.

In this paper, we design a new finiteness testing algorithm that
may be employed, for the first time, over any infinite field. The
algorithm is concise and practical. Our implementation is
distributed with {\sc Magma}, and we demonstrate that it performs
well for a range of inputs.

If a group $G$ is finite then, in practice, we can often construct
an isomorphic copy of $G$ over some finite field. As a
consequence, drawing on recent progress in computing with matrix
groups over finite fields \cite{CT,OBE}, we obtain the first
algorithms to answer many structural questions about $G$. These
include: computing $|G|$; testing membership in $G$; computing
Sylow subgroups, a composition series, and the solvable and
unipotent radicals of $G$.

We emphasize that this paper provides a framework for the solution
of broader computational problems than the testbed ones treated
here. SW-homomorphisms (defined below) are used in \cite{Draft} to
test nilpotency over certain fields. In \cite{Tits}, these are
extended to decide virtual properties of finitely generated linear
groups. The present paper gives a comprehensive account of our
techniques that is valid in all settings. For further discussion
of how these ideas have been developed, see the survey
\cite{Survey}.

Briefly, the paper is organized as follows.
Sections~\ref{CongruenceHomomorphisms} and
\ref{ConstructionofSWhomomorphisms} set up our computational
analogue of finite approximation. The algorithms are presented and
justified in Section~\ref{FinitenessAlgorithms}. In the final
section, we report on our {\sc Magma} implementation.

\section{Congruence homomorphisms of finitely generated linear
groups} \label{CongruenceHomomorphisms}

Let $\rho$ be a proper ideal of an (associative, unital) ring
$\Delta$. The natural surjection $\Delta \rightarrow\Delta/\rho$
induces an algebra homomorphism $\mathrm{Mat}(n, \Delta)
\rightarrow\mathrm{Mat}(n, \Delta/\rho)$, which restricts to a
group homomorphism $\mathrm{GL}(n, \Delta)
\rightarrow\mathrm{GL}(n, \Delta/\rho)$. All these congruence
homomorphisms will be denoted by $\phi_{\rho}$. The
\emph{principal congruence subgroup} $\Gamma_{\rho}$ is the kernel
of $\phi_{\rho}$ in $\mathrm{GL}(n, \Delta)$.

We fix some more notation, used throughout. Let $S = \{g_1,
\ldots, g_r \} \subseteq\mathrm{GL}(n, \mathbb{F})$, where
$\mathbb{F}$ is a field. Denote $\langle S\hspace*{.2mm} \rangle$
by $G$. Then $G \leq \mathrm{GL}(n, R)$, where $R\subseteq
\mathbb{F}$ is the (Noetherian) ring generated by the entries of
the matrices $g_i$, $g_i^{-1}$, $1 \leq i \leq r$. Recall that
$R/\rho$ is a finite field if $\rho$ is a maximal ideal (see
\cite[\mbox{p.} 50]{Wehrfritz}). For the purpose of studying $G$,
we may assume without loss of generality that $\mathbb{F}$ is the
field of fractions of $R$, and is a finitely generated extension
of its prime subfield.

Each finitely generated linear group possesses a normal subgroup
$N$ of finite index whose torsion elements are all unipotent; so
$N$ is torsion-free if $\mathrm{char}\,  R = 0$. A proof of this
result, due to Selberg (1960) and Wehrfritz (1970), can be found
in \cite[4.8, \mbox{p.} 56]{Wehrfritz}; a short new proof is
supplied by Proposition~\ref{Noetherian} and
Corollary~\ref{ElementarySelberg} below. We call such a normal
subgroup $N$ of a given linear group an \emph{SW-subgroup}. If
$\rho$ is an ideal of $R$ such that $\Gamma_{\rho}$ is an
SW-subgroup of $\mathrm{GL}(n,R)$, then $\phi_{\rho}$ is an
\emph{SW-homomorphism}. We now formulate conditions that enable us
to construct SW-homomorphisms.

\begin{proposition} \label{Noetherian}
Let $\Delta$ be a Noetherian integral domain, and $\rho$ be a
maximal ideal of $\Delta$. If $\Gamma_{\rho}$ has a non-trivial
torsion element $h$, then $p := \mathrm{char} (\Delta/\rho)>0$ and
$|h|$ is a power of $p$.
\end{proposition}
\begin{proof}
Set $|h|=q$. Since $\phi_\rho(h) = 1_n$, we have $h = 1_n + b$
where $b \in \mathrm{Mat}(n, \rho)$, $b\neq 0_n$. Hence $(1_n +
b)^q = 1_n$, so that
\[
\tag{$\dagger$} qb + {q\choose 2} b^2 + \cdots + b^q = 0_n.
\]
For $k\geq 2$, denote the $(i,j)$th entry of $b^k$ by
$b_{ij}^{(k)}$; then $b_{ij}^{(k)} \in \rho^k$. By the Krull
Intersection Theorem \cite[27.8, \mbox{p.} 437]{Isaacs}, $\cap_{k
= 1}^{\infty}\rho^k = \{0\}$. Hence there exists a positive
integer $c$ such that $b_{ij} \in \rho^c$ for all $i,j$, but
$b_{rs} \notin \rho^{c+1}$ for some $r, s$. This implies that
$b_{ij}^{(k)}\in \rho^{kc}$. Now
\[
qb_{rs} + {q\choose 2}b_{rs}^{(2)} + \cdots + b_{rs}^{(q)} = 0
\]
by ($\dagger$), so $qb_{rs}\in \rho^{2c} \subseteq \rho^{c+1}$.

Suppose that $q \notin \rho$. Since $\Delta/\rho$ is a field,
there exist $x\in \Delta$ and $y\in \rho$ such that $1=qx+y$. Then
\[
b_{rs} = qb_{rs}x + b_{rs}y
\]
and so, because $b_{rs} \in \rho^c$ and $qb_{rs}\in \rho^{c+1}$,
we get $b_{rs}\in \rho^{c+1}$. This contradiction proves that
$q\in \rho$. Thus $q$ must be a power of $p$. For if not, we could
have begun with $h$ of prime order different to $p$, and then
$\rho$ would contain two different prime integers and so would
contain $1$.
\end{proof}
\begin{corollary}\label{ElementarySelberg}
Let $\rho$, $\rho_1$, and $\rho_2$ be maximal ideals of the
Noetherian integral domain $\Delta$.
\begin{itemize}
\item[{\rm (i)}] If $\mathrm{char} (\Delta/\rho) = 0$ then
$\Gamma_\rho$ is torsion-free. \item[{\rm (ii)}] Suppose that
$\mathrm{char} \, \Delta = 0$, and $\mathrm{char} (\Delta/\rho_1)
\neq \mathrm{char} (\Delta/\rho_2)$. Then
$\Gamma_{\rho_1}\cap\Gamma_{\rho_2}$ is a torsion-free subgroup of
$\mathrm{GL}(n, \Delta)$. In particular, if $\Delta = R$ then
$\Gamma_{\rho_1}\cap\Gamma_{\rho_2}$ is an SW-subgroup of
$\mathrm{GL}(n, R)$. \item[{\rm (iii)}] Suppose that
$\mathrm{char} \, \Delta = p
> 0$. Then each torsion element of $\Gamma_\rho$ is unipotent. In
particular, if $\Delta = R$ then $\Gamma_{\rho}$ is an SW-subgroup
of $\mathrm{GL}(n, R)$.
\end{itemize}
\end{corollary}
\begin{proof} Clear from Proposition~\ref{Noetherian}.
\end{proof}
Note that parts (i) and (ii) of Corollary~\ref{ElementarySelberg}
contribute to a solution of the problem posed on \mbox{p.} 70 of
\cite{Supr}.

By Corollary~\ref{ElementarySelberg} (ii), if $\mathrm{char}\,R =
0$ then an SW-subgroup can be constructed as the intersection of
two congruence subgroups. Since this may not be convenient, we
mention one more result.
\begin{proposition}\label{notsquare}
Suppose that $\Delta$ is a Dedekind domain of characteristic zero,
and $\rho$ is a maximal ideal of $\Delta$ such that $\mathrm{char}
(\Delta/ \rho)=p>2$. If $p\notin \rho^{2}$ then $\Gamma_\rho$ is
torsion-free.
\end{proposition}
\begin{proof} See \cite[Theorem 4, \mbox{p.} 70]{Supr}.
\end{proof}

\section{Construction of SW-homomorphisms}
\label{ConstructionofSWhomomorphisms}

We now outline methods to construct both congruence homomorphisms
and SW-homomorphisms, given the assumptions on $\mathbb{F}$ made
in the second paragraph of Section~\ref{CongruenceHomomorphisms}.

Since $\mathbb{F}$ is a finitely generated extension of its prime
subfield, there is a subfield $\mathbb{P}\subseteq \mathbb{F}$ of
finite degree over the prime subfield, and elements $x_1, \ldots,
x_m$ ($m\geq 0$) algebraically independent over $\mathbb{P}$, such
that $\mathbb F$ is a finite extension of $\mathbb{L}=
\mathbb{P}(x_1, \ldots, x_m)$; say $|\mathbb{F}:\mathbb{L}| =
e\geq 1$. Here $|\mathbb{P}:\mathbb{Q}| = k \geq 1$ if
$\mathrm{char}\, \mathbb{F} = 0$, and if $\mathrm{char}\,
\mathbb{F} = p>0$ then $\mathbb{P}$ is the field $\mathbb{F}_{q}$
of size $q$.

Each type of field is considered in its own section below. For an
integral domain $\Delta$ and $\mu\in \Delta \setminus \{ 0\}$, let
$\frac{1}{\mu}\Delta$ denote the ring of fractions with
denominators in the multiplicative submonoid of $\Delta$ generated
by $\mu$.

\subsection{The rational field}
\label{RationalField}

Let $\mathbb{F} = \mathbb Q$. Then $R=\frac{1}{\mu}\mathbb Z$
where $\mu$ is the least common multiple of the denominators of
the entries in the matrices $g_i, g_i^{-1}$, $1\leq i\leq r$. For
a prime $p \in \mathbb Z$ not dividing $\mu$, define $\phi_1 =
\phi_{1,p}: \mathrm{GL}(n, R) \rightarrow \mathrm{GL}(n,p)$ to be
entry-wise reduction modulo $p$. If $p>2$ then we denote
$\phi_{1,p}$ by $\Phi_1= \Phi_{1,p}$. By
Proposition~\ref{notsquare}, $\Phi_1$ is an SW-homomorphism.

\subsection{Number fields}
\label{AlgebraicNumberFieldsSection}

Let $\mathbb{F}$ be a number field, so that $\mathbb{F} = \mathbb
Q(\alpha)$ for some algebraic number $\alpha$. Let $f(t)$ be the
minimal polynomial of $\alpha$, of degree $k$. Multiplying
$\alpha$ by a common multiple of the denominators of the
coefficients of $f(t)$, if necessary, we may assume that $\alpha$
is an algebraic integer; that is,  $f(t) \in \mathbb{Z}[t]$.

We have $R \subseteq \frac{1}{\mu}\mathbb Z[\alpha] \subseteq
\frac{1}{\mu}\mathcal{O}$ for some $\mu \in \mathbb Z$, where
$\mathcal{O}$ is the ring of integers of $\mathbb{F}$. We define
an SW-homomorphism on $R$ as the restriction of a congruence
homomorphism on the Dedekind domain $\frac{1}{\mu}\mathcal{O}$.

Let $p \in \mathbb Z$ be a prime not dividing $\mu$, and denote by
$\bar{f}(t)$ the polynomial obtained by mod $p$ reduction of the
coefficients of $f(t)$. Further, let $\bar{\alpha}$ be a root of
$\bar{f}(t)$, so that $\bar{\alpha}$ is a root of some
$\mathbb{Z}_p$-irreducible factor $\bar{f}_j(t)$ of $\bar{f}(t)$.
Each $b\in R$ may be expressed uniquely in the form $b =
\sum^{k-1}_{i = 0}c_i\alpha^i$ where $c_i \in \frac{1}{\mu}
\mathbb Z$. Thus the assignment $\phi_{2,p} : b \mapsto
\sum^{k-1}_{i=0}\phi_{1,p}(c_i)\bar{\alpha}^i$ is well-defined.
Moreover, $\phi_{2,p}$ is a ring homomorphism $R\rightarrow
\mathbb Z_p(\bar{\alpha})= \mathbb F_{p^l}$, say. Thus we have an
induced congruence homomorphism $\phi_{2,p}: \mathrm{GL}(n, R)
\rightarrow\mathrm{GL}(n, p^{l})$.

Next, we state criteria under which $\phi_2 = \phi_{2,p}$ is an
SW-homomorphism.
\begin{lemma} \label{DiscriminantHypothesis}
Suppose that $p\in \Z$ is an odd prime dividing neither $\mu$ nor
the discriminant of $f(t)$. Then the kernel of $\phi_{2,p}$ on
$\mathrm{GL}(n, R)$ is torsion-free.
\end{lemma}
\begin{proof}
Let $f_j(t)$ be a preimage of $\bar{f}_j(t)$ in $\mathbb Z[t]$.
The ideal $\rho$ generated by $p$ and $f_j(\alpha)$ in
$\frac{1}{\mu}\mathcal{O}$ is maximal, by \cite[Theorem
3.8.2]{Koch}. Hence $\rho \cap R$ is a maximal ideal of $R$. Also
$p\not \in \rho^2$ by \cite[Proposition 3.8.1, Theorem
3.8.2]{Koch}. The lemma then follows from
Proposition~\ref{notsquare}.
\end{proof}

\begin{lemma}\label{ThreePointThree}
There are no non-trivial $p$-subgroups of $\mathrm{GL}(n,
\mathbb{F})$ if $p>nk + 1$.
\end{lemma}
\begin{proof}
Let $g\in \mathrm{GL}(n, \mathbb{Q})$ be of order $p$. Since the
characteristic polynomial of $g$ has a primitive $p$th root of
unity as a root, it is divisible by the $p$th cyclotomic
polynomial. Thus $p-1 \leq n$. The general claim holds because
each subgroup of $\mathrm{GL}(n,\mathbb F)$ is isomorphic to a
subgroup of $\mathrm{GL}(nk,\mathbb Q)$.
\end{proof}

\begin{corollary} \label{TorsionFreePrimeLowerBound}
Suppose that $\Delta$ is a Noetherian subring of $\mathbb{F}$, and
$\rho$ is a maximal ideal of $\Delta$ such that
$\mathrm{char}(\Delta/\rho)=p>nk+1$. Then $\Gamma_\rho$ is
torsion-free.
\end{corollary}
\begin{proof}
This is a consequence of Proposition~\ref{Noetherian} and
Lemma~\ref{ThreePointThree}.
\end{proof}

Let $p\in \mathbb Z$ be a prime not dividing $\mu$. We denote
$\phi_{2,p}$ by $\Phi_2 = \Phi_{2,p}$ if one of the following
extra conditions on $p$ is satisfied: $p$ is odd and does not
divide the discriminant of the minimal polynomial of $\alpha$; or
$p>nk+1$. The preceding discussion shows that $\Phi_{2}$ is an
SW-homomorphism.

\begin{example}
Suppose that $\mathbb{F}$ is a cyclotomic field, say $\mathbb{F} =
\Q(\zeta)$ where $\zeta$ is a primitive $c$th root of unity,
$c>2$. If $p>2$ and $p$ does not divide $\mathrm{lcm}(\mu, c)$,
then $\phi_{2,p}$ is an SW-homomorphism by
Lemma~\ref{DiscriminantHypothesis}.
\end{example}

\subsection{Function fields}
\label{FFieldSubsection}

Let $\mathbb F = \mathbb P(x_1, \ldots, x_m)$, $m \geq 1$, where
$\mathbb{P}$ is $\Q$, a number field, or $\mathbb{F}_q$. We have
$R \subseteq \frac{1}{\mu}\mathbb P[x_1, \ldots, x_m]$ for some
$\mu = \mu(x_1, \ldots, x_m)$ determined by $S\cup S^{-1}$.

Let $a = (a_1, \ldots, a_m)$ be a non-root of $\mu$. If
$\mathrm{char}\, \mathbb{F} = 0$, then $a_i\in \mathbb{P}$ for all
$i$; if $\mathbb F$ has positive characteristic, then the $a_i$
are in $\mathbb P$ or some finite extension. Define $\phi_3 =
\phi_{3,a}$ to be the map that substitutes $a_i$ for $x_i$, $1
\leq i \leq m$. Corollary~\ref{ElementarySelberg} (i) implies that
$\phi_3: \mathrm{GL}(n, R) \rightarrow \mathrm{GL}(n, \mathbb P)$
is a homomorphism with torsion-free kernel if $\mathrm{char}\,
\mathbb{F} = 0$. We then obtain an SW-homomorphism in zero
characteristic by setting $\Phi_3 = \Phi_{3,a,p}=\Phi_{i,p} \circ
\phi_{3,a}$, where $i =1$ or $2$ if $\mathbb P = \mathbb Q$ or
$\mathbb{P}$ is a number field, respectively. If $\mathbb{P} =
\mathbb{F}_q$ then $\Phi_3 = \phi_{3}$ is an SW-homomorphism by
Corollary~\ref{ElementarySelberg} (iii). Notice that
$\Phi_{3,a,p}$ is defined for all but a finite number of $a$ and
$p$ when $m=1$; otherwise, $\Phi_{3,a,p}$ is defined for
infinitely many $a$ and $p$.

\subsection{Algebraic function fields}
\label{AlgebraicFunctionFields}

For $m\geq 1$, let $\mathbb L =\mathbb P(x_1, \ldots, x_m)$ and
$\mathbb{L}_0= \mathbb{P}[x_1, \dots, x_m]$, where again
$\mathbb{P}$ is $\Q$, a number field, or $\mathbb{F}_q$. We assume
that $\mathbb F = \mathbb L(\alpha)$ is a simple extension of
$\mathbb L$ of degree $e>1$. For instance, we can stipulate that
$\mathbb{F}$ is a separable extension of $\mathbb L$ (e.g., in
characteristic $p$ this is assured if $p\nmid e$). Let $f(t)\in
\mathbb L_0[t]$ be the minimal polynomial of $\alpha$. We have $R
\subseteq \frac{1}{\mu}\mathbb L_0[\alpha]$ for some $\mu \in
\mathbb L_0$ determined in the usual way by the input $S$.

Suppose that $a=(a_1, \ldots, a_m)$ is a non-root of $\mu$, where
the $a_i$ are in $\mathbb{P}$ or a finite extension. Denote by
$\bar f(t)$ the polynomial obtained by substitution of $a$ in the
coefficients of $f(t)$. Define $\bar{c} = \phi_{3,a}(c)$ for $c\in
\frac{1}{\mu}\mathbb{L}_0$ similarly. Let $\bar{\alpha}$ be a root
of $\bar{f}(t)$. Define $\phi_4=\phi_{4,a}: R\rightarrow \mathbb
P(\bar {\alpha})$ by $\phi_4: \sum_{i=0}^{e-1}c_i\alpha^i \mapsto
\sum_{i = 0}^{e-1}\bar {c}_i\bar {\alpha}^i$. Therefore, if
$\rm{char}\, \mathbb{F} = 0$ then we get an induced congruence
homomorphism $\phi_4: \mathrm{GL}(n, R) \rightarrow\mathrm{GL}(n,
\mathbb P(\bar {\alpha}))$, whose kernel is torsion-free by
Corollary~\ref{ElementarySelberg} (i). Set $\Phi_4=\Phi_{4,a,p} =
\Phi_{i,p} \circ \phi_{4,a}$, where $i=1$ if $\mathbb P(\bar
{\alpha}) = \Q$, and $i=2$ if $\mathbb P(\bar {\alpha})$ is a
number field. If $\mathrm{char}\, \mathbb{F}
> 0$ then we set $\Phi_{4} = \phi_{4}$. In all cases $\Phi_4$
is an SW-homomorphism. As with $\Phi_{3,a,p}$, the homomorphism
$\Phi_{4,a,p}$ is defined for infinitely many $a$ and $p$, and for
all but a finite number of $a$, $p$ when $m=1$.

\begin{remark}
Fields $\F$ as in
Sections~\ref{RationalField}--\ref{AlgebraicFunctionFields} are
the main ones supported by {\sf GAP} and {\sc Magma}.
\end{remark}

\begin{remark}
SW-homomorphisms are used in \cite[Section 5.3]{Tits} to test
whether $G\leq \mathrm{GL}(n,\F)$ is central-by-finite; indeed,
each `W-homomorphism' defined in that paper is a special kind of
SW-homomorphism. They also feature in the nilpotency testing
algorithm of \cite{Draft}.
\end{remark}

\subsection{Analyzing congruence homomorphisms }
\label{AnalyzingCongruenceHomomorphismsSubsection}

We now prove some results that will be helpful in the analysis of
our algorithms.

\begin{lemma}\label{ForAllButaFiniteNumber}
Let $\Delta$ be a Dedekind domain, and let $G$ be a finitely
generated subgroup of $\mathrm{GL}(n, \Delta)$. For all but a
finite number of maximal ideals $\rho$ of $\Delta$, the following
are true:
\begin{itemize} \item[{\rm (i)}] if $G$ is finite then
$\phi_{\rho}$ is an isomorphism of $G$ onto $\phi_\rho(G)$;
\item[{\rm (ii)}] if $G$ is infinite, and $\nu$ is a positive
integer, then $\phi_{\rho}(G)$ contains an element of order
greater than $\nu$.
\end{itemize}
\end{lemma}
\begin{proof}(\mbox{Cf.} \cite[\mbox{p.} 51]{Wehrfritz} and \cite[Lemma
3]{JSC4534}.) Note that a non-zero element $a$ of $\Delta$ is
contained in only finitely many maximal ideals of $\Delta$. To see
this, let $a\Delta = \rho_{1}^{e_1}\cdots \rho_{c}^{e_c}$, where
the $\rho_i$ are maximal ideals. If $\rho$ is a maximal ideal of
$\Delta$ containing $a$, then $\rho_{1}^{e_1}\cdots \rho_{c}^{e_c}
\subseteq \rho$, so $\rho=\rho_i$ for some $i$.

Next, let $M=\{h_1, \ldots , h_d \} \subseteq \mathrm{Mat}(n,
\Delta)$, and for each pair $l,k\in\{1, \ldots,d\}$, $l\neq k$,
choose $(i,j)$ such that $h_l(i,j)- h_k(i,j)\neq 0$. Denote the
product of all differences $h_l(i,j)- h_k(i,j)$ by $a_M$. If
$\rho$ is an ideal of $\Delta$ not containing $a_M$, then
$|\phi_{\rho}(M)| = |M|$.

Taking $M$ to be the set of elements of $G$,  part (i) is now
clear.

If $G$ is infinite then $G$ contains an element $g$ of infinite
order, by a result of Schur \cite[Theorem~5, \mbox{p.} 181]{Supr}.
Thus, taking $M$ to be $\{g,\ldots, g^{\nu}, g^{\nu + 1}\}$, we
get (ii).
\end{proof}

To utilize Lemma~\ref{ForAllButaFiniteNumber} in our context, let
$\mathbb{F}$ be one of $\mathbb{Q}$, a number field,
$\mathbb{P}(x)$, or a finite extension of $\mathbb{P}(x)$. The
relevant SW-homomorphism $\Phi$ on $\mathrm{GL}(n,R)$ is the
restriction of a congruence homomorphism $\phi_\rho$ on
$\mathrm{GL}(n,\Delta)$, where $\Delta$ is a Dedekind domain with
maximal ideal $\rho$. Hence for $G \leq \mathrm{GL}(n,R)$ and all
but a finite number of choices in the definition of $\phi_\rho$,
the following hold: (a)~if $G$ is finite, then $\Phi$ is an
isomorphism on $G$; (b)~if $G$ is infinite, then $\Phi(G)$
contains an element of order greater than any given positive
integer $\nu$. For the other fields $\mathbb{F}$ where $R$ may not
be contained in a Dedekind domain (function fields with more than
one indeterminate, or finite extensions thereof), it is still true
that there are infinitely many SW-homomorphisms $\Phi$ such that
(a) and (b) hold. This follows from the definition of $\Phi$ in
each case, and arguing as in the proof of
Lemma~\ref{ForAllButaFiniteNumber}.

\section{Finiteness algorithms for matrix groups}
\label{FinitenessAlgorithms}

\subsection{Preliminaries: asymptotic bounds}
\label{asymptotics}

We continue with the notation of the previous section:
$|\mathbb{F}:\mathbb{L}|=e\geq 1$, $\mathbb{L}= \mathbb{P}(x_1,
\ldots, x_m)$, $m\geq 0$, and $|\mathbb{P}:\mathbb{Q}|=k \geq 1$
or $\mathbb{P} = \mathbb{F}_q$.

Suppose first that $\mathrm{char}\, \mathbb{F} = 0$. Put
$n_0=nke$.
\begin{lemma}\label{FiniteIsomorphic}
A finite subgroup $G$ of $\mathrm{GL}(n, \mathbb F)$ is isomorphic
to a subgroup of $\mathrm{GL}(n_0, \mathbb{Q})$.
\end{lemma}
\begin{proof}
Certainly $G$ is isomorphic to a subgroup of $\mathrm{GL}(ne,
\mathbb{L})$, and a subgroup of $\mathrm{GL}(ne, \mathbb{P})$ is
isomorphic to a subgroup of $\mathrm{GL}(nke, \mathbb{Q})$. The
lemma follows from \cite[\mbox{p.} 69, Corollary~4]{Supr}.
\end{proof}

It is well-known that the order of a finite subgroup of
$\mathrm{GL}(n, \mathbb Q)$ is bounded by a function of $n$ (see,
e.g., \cite{Feit, Friedland}). Hence by
Lemma~\ref{FiniteIsomorphic} there are functions $\nu_1 =
\nu_1(n_0)$ and $\nu_2 = \nu_2(n_0)$ bounding the order of a
finite subgroup of $\mathrm{GL}(n, \mathbb{F})$ and the order of a
torsion element of $\mathrm{GL}(n, \mathbb{F})$, respectively. For
$n_0>10$ or $n_0=3,5$ we may take $\nu_1 = 2^{n_0}(n_0)!$ by
\cite[Theorem A]{Feit}; for the remaining $n_0$, values of $\nu_1$
are also listed there. A suitable function $\nu_2$ is given by the
next lemma.
\begin{lemma}\label{ExponentBound}
%
If $g$ is a torsion element of $\mathrm{GL}(n, \mathbb F)$ then
$|g| \leq 2^{\lfloor \log_2
n_0\rfloor + 1} 3^{\left \lfloor n_0/2 \right \rfloor}$.
\end{lemma}
\begin{proof}
Let  $\mathbb{F} = \mathbb{Q}$. If $|g|$ is odd then $|g| \leq
3^{\left \lfloor n/2 \right \rfloor}$ by \cite[\mbox{p.}
3519]{Friedland}. Suppose that $g$ is a $2$-element. Then $g$ is
conjugate to a monomial matrix over $\Q$ (see
\cite[IV.4]{PleskenCRLG}). Since the order of a $2$-element in
$\mathrm{Sym}(n)$ is bounded by the largest power $2^t$ of $2$
less than or equal to $n$, 
$|g| \leq 2^{t+1}$. Lemma~\ref{FiniteIsomorphic} now
implies the result in the general case $\mathbb{F}\supseteq
\mathbb{Q}$.
\end{proof}
Here is one more useful condition to detect infinite groups in
characteristic zero.
\begin{lemma}\label{OneMoreUsefulCondition}
If $G \leq \mathrm{GL}(n, \mathbb F)$ is finite and $p > n_0 + 1$
then $p \nmid |G|$.
\end{lemma}
\begin{proof}
This follows from Lemmas~\ref{ThreePointThree} and
\ref{FiniteIsomorphic}.
\end{proof}

Now suppose that $\mathrm{char}\, \mathbb{F} > 0$. The order of a
finite subgroup of $\mathrm{GL}(n, \mathbb{F})$ can be arbitrarily
large. On the other hand, the orders of torsion elements of
$\mathrm{GL}(n, \mathbb{F})$ are bounded. The next lemma furnishes
such a bound.
\begin{lemma} \label{excelbound}
Let $n_0 = ne$. If $g$ is a torsion element of
$\mathrm{GL}(n,\mathbb{F})$ then $|g|\leq q^{n_0}-1$.
\end{lemma}
\begin{proof}
The proof is essentially the same as that of \cite[Theorem 3.3,
Corollary 3.4]{rtb99}. We recap the main points. It suffices to
assume that $\mathbb{F} = \mathbb{L}$. By \cite{Zalesskii}, $g$ is
conjugate to a block upper triangular matrix, where the
(irreducible) blocks are $\mathbb{F}_q$-matrices. Hence the
characteristic polynomial of $g$ has $\mathbb{F}_q$-coefficients.
It follows that the dimension of $\gp{g}_{\mathbb{F}_q}$ is at
most $n$, and so every invertible element of this enveloping
algebra has order at most $q^n-1$.
\end{proof}

\subsection{Testing finiteness}\label{testingfiniteness}

By Section~\ref{ConstructionofSWhomomorphisms}, we are able to
construct a congruence image $\phi_\rho(G)$ of $G\leq$
$\mathrm{GL}(n,\F)$ over a finite field such that the torsion
elements of $G_\rho:= G\cap \Gamma_\rho$ are unipotent. Thus, to
decide finiteness of $G$, we merely test whether $G_\rho$ is
trivial ($\mathrm{char}\, \mathbb{F} = 0$), or whether $G_{\rho}$
is unipotent ($\rm{char}\, \mathbb{F} > 0$). Both tasks can be
accomplished with just {\it normal generators} of $G_\rho$:
generators for a subgroup whose normal closure in $G$ is
$G_\rho$. That is, we do not need to construct the full
congruence subgroup. Normal generators are found by a standard
method \cite[\mbox{pp.} 299--300]{HoltEicketal05} that requires a
presentation of $\phi_\rho(G)$ as input. Since it is a matrix
group over a finite field, we can compute a presentation of
$\phi_\rho(G)$ using the algorithms described in \cite{CT,OBE}. We
refer to such an algorithm as $\tt Presentation$. Let $\tt
SWImage$ be an algorithm that constructs a congruence image over a
finite field. The congruence homomorphism in question is one of
the SW-homomorphisms $\Phi=\Phi_i$, $1\leq i \leq 4$, defined in
Sections~\ref{RationalField}--\ref{AlgebraicFunctionFields}. The
following procedure tests finiteness along the lines just
explained (see Section~\ref{asymptotics} for definitions of $n_0$
and $\nu_1$).

\bigskip

$\tt IsFiniteMatrixGroup$

\vspace*{1mm}

Input: $S = \{g_1, \ldots, g_r\}\subseteq \mathrm{GL}(n, \mathbb
F)$.

Output: $\tt true$ if $G =\gp{S\,}$ is finite; $\tt false$
otherwise.

\vspace*{1mm}

\begin{enumerate}
\item $H := {\tt SWImage}(G) = \gp{\Phi(g_1), \ldots, \Phi(g_r)}$.

\item If $\mathrm{char} \,\mathbb{F}=0$ and either $|H|>\nu_1$ or
$p$ divides $|H|$ for some prime $p>n_0+1$, then return $\tt
false$.

\item ${\tt Presentation}(H) :=\group{\Phi(g_1), \ldots,
\Phi(g_r)}{\omega_j(\Phi(g_1), \ldots, \Phi(g_r)) = 1; 1\leq j\leq
t}$.

\item  $K := \{ \omega_j (g_1, \ldots , g_r) \mid  1\leq j \leq t
\}$.

\item If $\mathrm{char} \, \mathbb{F}=0$ and $K= \{1_n\}$, or
$\mathrm{char} \, \mathbb{F}>0$ and ${\tt IsUnipotent}(\gp{K}^G)$,
then return $\tt true$. Else return $\tt false$.
\end{enumerate}

\bigskip

Step (2) is justified by Lemma~\ref{OneMoreUsefulCondition} and
the comments before Lemma~\ref{ExponentBound}. For example, if
$\mathbb{F}$ is a number field then Lemma
\ref{ForAllButaFiniteNumber} suggests that the initial check in
this step will usually identify that $G$ is infinite. We test
unipotency of the congruence subgroup $\gp{K}^G$ in step (5) using
the normal generating set $K$. A procedure for doing this, based
on computation in enveloping algebras, is given in \cite[Section
5.2]{Tits}. Also note that we can apply a conjugation isomorphism
as in \cite{GlasbyHowlett} to write the SW-image over the smallest
possible finite field of the chosen characteristic.

Next we consider the special but very important case that $G$ is a
cyclic group: testing whether $g\in \mathrm{GL}(n, \mathbb{F})$
has finite order. Let $\nu_2$ be an upper bound on the order of a
torsion element of $\mathrm{GL}(n,\mathbb{F})$. See
Lemmas~\ref{ExponentBound} and \ref{excelbound} for values of
$\nu_2$.

\bigskip

$\tt IsFiniteCyclicMatrixGroup$

\vspace*{1mm}

Input: $g \in \mathrm{GL}(n, \mathbb F)$.

Output: $\tt true$ if $g$ has finite order; $\tt false$ otherwise.

\vspace*{1mm}

\begin{enumerate}

\item $h := {\tt SWImage}(g)$.

\item $d := {\tt Order}(h)$.

\item If $d > \nu_2$, or $\mathrm{char}\, \mathbb{F}=0$ and $p\mid
d$ for some prime $p>n_0 + 1$, then return $\tt false$.

\item If $\mathrm{char}\, \mathbb{F}=0$ and $g^d = 1_n$, or
$\mathrm{char}\, \mathbb{F}>0$ and ${\tt IsUnipotent}(g^d)$, then
return $\tt true$. Else return $\tt false$.

\end{enumerate}

\bigskip

Note that $g^d$ is unipotent in characteristic $p>0$ if and only
if its order divides $p^{\left \lceil \log_p n\right \rceil}$ (see
\cite[\mbox{p.} 192]{Supr}). Also, if $\mathrm{char}\,
\mathbb{F}=0$ and $\tt IsFiniteCyclicMatrixGroup$ returns $\tt
true$, then the order $d$ of $g$ is calculated in step (2). In the
situations covered by Lemma~\ref{ForAllButaFiniteNumber}, if $|g|$
is infinite then $d>\nu_2$ for all but a finite number of choices
of $\Phi$. That is, we expect that infiniteness of $|g|$ will be
detected at step (3) of $\tt IsFiniteCyclicMatrixGroup$.

Recall that an infinite group $G\leq \mathrm{GL}(n,\F)$ has an
infinite order element. Hence, as a precursor to running $\tt
IsFiniteMatrixGroup$, we check via $\tt IsFiniteCyclicMatrixGroup$
whether `random' elements of $G$, produced by a variation of the
product replacement algorithm \cite{PRA}, have infinite order;
\mbox{cf.} \cite[Section 8.2]{BBR}.

\subsection{Recognizing finite matrix groups}
\label{recogfinite}

Suppose that $G\leq \mathrm{GL}(n,\mathbb{F})$ is finite. We
describe how to find an isomorphic copy of $G$ in some
$\mathrm{GL}(n,q)$ and carry out further computations with $G$.

If $\mathrm{char}\, \mathbb{F}=0$ then ${\tt SWImage}(G)= \Phi(G)$
is isomorphic to $G$. If $\mathrm{char}\, \mathbb{F}>0$ then the
congruence subgroup may be non-trivial. We repeat the construction
of normal generators of the congruence subgroup for different
choices of $\Phi$, until we find a $\Phi$ for which all these
generators are trivial. By the discussion at the end of
Section~\ref{AnalyzingCongruenceHomomorphismsSubsection}, if $m =
1$ (there is just one indeterminate) then in a finite number of
iterations we will get an isomorphic copy of $G$ by
Lemma~\ref{ForAllButaFiniteNumber}. Otherwise, although there are
infinitely many isomorphisms $\Phi$, the procedure may not
terminate. In our many experiments the procedure always succeeded
in finding an isomorphic copy of $G$.

Once we have an isomorphic copy, algorithms for matrix groups over
finite fields (see \cite{CT} and \cite[Chapter
10]{HoltEicketal05}) are used to investigate the structure and
properties of $G$. In particular, we can

\begin{itemize}
\item compute a composition series and short presentation for $G$;

\item compute $|G|$;

\item compute the solvable and unipotent radicals, the derived
subgroup, center, and Sylow subgroups of $G$;

\item test membership of $x\in \mathrm{GL}(n,\F)$ in $G$.
\end{itemize}
Where feasible, the computation is undertaken directly in the
isomorphic copy, and the result is `lifted' by means of the known
isomorphism to $G$. Sometimes this involves additional work. For
instance, membership testing requires that we construct a new
isomorphic copy; namely, of $\gp{G, x}$.

\section{Implementation and performance}
\label{expresultssection}

The algorithms have been implemented in {\sc Magma} as part of our
package {\sc Infinite} \cite{Infinite}. We use machinery from the
{\sc CompositionTree} package \cite{CT,OBE} to study congruence
images and construct their presentations.

We implemented SW-homomorphisms as per
Sections~\ref{RationalField}--\ref{AlgebraicFunctionFields}. These
are applied in {\sc Infinite} to solve specific problems, such as
testing finiteness, virtual properties, and nilpotency (the latter
over an arbitrary field, significantly enhancing \cite{Draft}).
Here we report on the algorithms of
Sections~\ref{testingfiniteness} and \ref{recogfinite}.

In our implementation of $\tt IsFiniteMatrixGroup$ and $\tt
IsFiniteCyclicMatrixGroup$, we construct (at least) two
SW-homomorphisms and determine the orders of the images of $G$
under these. If $G$ is finite and $\mathrm{char}\, \F = 0$, then
the orders must be identical. In positive characteristic, the
least common multiple of the orders of two images of an element of
finite $G$ must be at most $\nu_2$. The single most expensive task
is evaluating relations to obtain normal generators for the kernel
of an SW-homomorphism, since this may lead to blow-up in the size
of matrix entries. Hence we first check the orders of images under
several SW-homomorphisms before we evaluate relations.

In \cite{PositiveFiniteness} we proposed an alternative algorithm
to decide finiteness for groups defined over function fields of
positive characteristic. This is an option in {\sc Infinite}; it
avoids evaluation of relations over the field of definition, and
is sometimes faster than $\tt IsFiniteMatrixGroup$ for such
groups.

We now describe sample outputs that illustrate the efficiency and
scope of our implementation. The examples chosen cover the main
domains and a variety of groups. Our experiments were performed
using {\sc Magma} V2.17-2 on a 2GHz machine. All examples are
randomly conjugated, so that generators are not sparse, and matrix
entries (numerators and denominators) are large. Since random
selection plays a role in some of the {\sc CompositionTree}
algorithms, times stated are averages over three runs. The
complete examples are available in the {\sc Infinite} package.

\begin{enumerate}
\item $G_1\leq \mathrm{GL}(24,\Q(\zeta_{17}))$ is a conjugate of
the monomial group $\gp{\zeta_{17}} \wr \mathrm{Sym}(24)$. It has
order $17^{24}24!$, the maximum possible for a finite subgroup of
$\mathrm{GL}(24,\Q(\zeta_{17}))$ by \cite{Feit}. We decide
finiteness of this $3$-generator group and determine its order in
$1435$s; compute a Sylow $3$-subgroup in $22$s; and the derived
group in $57$s.

\item $G_2\leq \mathrm{GL}(12,\F)$ where $\F=\mathbb{P}(x)$ and
$\mathbb{P} = \Q(\sqrt{2})$. It is conjugate to $H_1\wr H_2$ where
$H_1$ is ${\tt RationalMatrixGroup}(4, 2)$ and $H_2={\tt
PrimitiveSubgroup}(3, 1)$, both from standard {\sc Magma}
databases. We decide finiteness of this $7$-generator group in
$18$s; compute its order $2^{16} 3^7$ in $1435$s; its centre in
$3$s; and its Fitting subgroup in $3$s.

\item $G_3 \leq \mathrm{GL}(20,\mathbb{F})$ where $\mathbb{F}$ is
a degree $2$ extension of the function field $\Q(x)$. It is
conjugate to the derived subgroup of the monomial group
$\gp{-1}\wr \mathrm{Sym}(20)$ in $\mathrm{GL}(20,\F)$. We decide
finiteness and compute the order of this $31$-generator group in
$1090$s; and construct a Sylow $7$-subgroup in $5$s.

\item $G_4 \leq \mathrm{GL}(100,\Q(\zeta_{19}))$. We prove that
this $14$-generator group is infinite in $9$s.

\item $G_5 \leq \mathrm{GL}(30,\F)$ where $\F$ is an algebraic
function field of degree $3$ over $\Q(x)$. We prove that this
$4$-generator group is infinite in $1024$s.

\item $G_6 \leq \mathrm{GL}(6, \F)$ where $\F$ is an algebraic
function field of degree $2$ over $\mathbb{F}_{9}(x)$. It is
conjugate to $\mathrm{GL}(6,3^2)$.  We find the order of this
$2$-generator group in $18$s; its unipotent radical in $15$s; a
Sylow $3$-subgroup $H$ in $18$s; and compute the normalizer in
$G_6$ of $H$ in $42$s.

\item $G_7 \leq \mathrm{GL}(16,\F)$ where $\F$ is a degree $3$
extension of $\mathbb{F}_{2}(x)$. It is conjugate to the Kronecker
product of $\mathrm{GL}(8,2)$ with a unipotent subgroup of
$\mathrm{GL}(2,\mathbb{F}_{2}(x))$. We decide finiteness of this
8-generator group in $16$s; we compute its order $16 \cdot
|\mathrm{GL}(8, 2)|$ and an isomorphic copy in $488$s; and
determine the Fitting subgroup in $12$s.

\item $G_8 \leq \mathrm{GL}(12,\mathbb{F})$ where $\mathbb{F}$ is
a function field with two indeterminates over $\mathbb{F}_5$. We
prove that this $8$-generator group is infinite in $6$s.

\item $G_9 \leq \mathrm{GL}(12,\mathbb{F})$ where $\mathbb{F}$ is
a degree $2$ extension of a univariate function field over
$\mathbb{F}_5$. We prove that this $8$-generator group is infinite
in $10$s.
\end{enumerate}

\bibliographystyle{amsplain}

\end{document}